\documentclass[12pt]{amsart}
\usepackage{amssymb}
\usepackage{amsmath, amscd}
\usepackage{amsthm}
\usepackage{amstext}
\usepackage[table,xcdraw]{xcolor}
\usepackage[colorlinks=true, linkcolor=blue,urlcolor=blue]{hyperref}

\newtheorem{theorem}{Theorem}[section]

\newtheorem{proposition}[theorem]{Proposition}
\newtheorem{lemma}[theorem]{Lemma}

\newtheorem{corollary}[theorem]{Corollary}
\theoremstyle{definition}

\newtheorem{example}[theorem]{Example}
\newtheorem{definition}[theorem]{Definition}

\begin{document}

\author[P. Danchev]{Peter Danchev}
\address{Institute of Mathematics and Informatics, Bulgarian Academy of Sciences, 1113 Sofia, Bulgaria}
\email{danchev@math.bas.bg; pvdanchev@yahoo.com}	

\author[G. Karamali]{Gholamreza Karamali}
\address{Department of Basic Sciences, Shahid Sattari Aeronautical University of Science and Technology, Tehran, Iran}
\email{rezakaramali918@gmail.com}

\author[O. Hasanzadeh]{Omid Hasanzadeh}
\address{Department of Basic Sciences, Shahid Sattari Aeronautical University of Science and Technology, Tehran, Iran}
\email{hasanzadeomiid@gmail.com}

\author[M. Esfandiar]{Mehrdad Esfandiar}
\address{Department of Mathematics and Computer Science Shahed University Tehran, Iran}
\email{mehrdad.esfandiar@shahed.ac.ir}

\title[Strongly \( J^{\#} \)-Clean Rings]{On Strongly \( J^{\#} \)-Clean Rings}

\maketitle

\date{\today}

\begin{abstract} We define and examine the class of {\it strongly \( J^{\#} \)-clean rings} consisting of those rings $R$ such that each element of $R$ is the sum of an idempotent from $R$ and an element from $J^{\#}(R)$ that commute with each other. More exactly, we prove that these rings are simultaneously strongly clean and Dedekind-finite as well as that they factor-ring modulo the Jacobson radical is always Boolean, and also provide some close relations with certain other well-established classes of rings like these of local, semi-local and strongly J-clean rings (as introduced by Chen on 2010) showing the surprising fact that the classes of strongly \( J^{\#} \)-clean and strongly J-clean rings, actually, do coincide. Moreover, a few more extensions of the newly defined class such as group rings and generalized matrix rings are provided too.

\end{abstract}

\section{Introduction and Fundamentals}

Throughout the present article, all rings are assumed to be unital and associative. Almost all symbols and concepts are traditional being consistent with the well-known books \cite{lamf} and \cite{lame}. Standardly, the Jacobson radical, the set of nilpotent elements, the set of idempotent elements, the set of quasi-nilpotent elements and the set of units of \( R \) are designated by \( J(R) \), \( Nil(R) \), \( Id(R) \), $QN(R)$ and \( U(R) \), respectively. Additionally, we write $M_n(R)$ and $T_n(R)$, respectively, for the $n \times n$ matrix ring and the $n \times n$ upper triangular matrix ring. Certainly, a ring is termed {\it abelian} if each its idempotent element is central.

Mimicking \cite{2} and \cite{7}, an element $r \in R$ is said to be {\it clean} if there is an idempotent $e \in R$ and an unit $u \in R$ such that $r=e+u$. Such an element $r$ is further called {\it strongly clean} if the existing idempotent and unit can be taken such that $ue=eu$. Accordingly, a ring is called {\it clean} (respectively, {\it strongly clean}) if each of its elements is clean (respectively, strongly clean). For a comprehensive investigation of this class of rings are written too many papers as, for a more detailed account, we refer to these cited in the given at the end of text reference list.

On the other hand, imitating \cite{diesl}, an element $r\in R$ is said to be {\it nil-clean} if there is an idempotent $e \in R$ and a nilpotent $b \in R$ such that $r=e+b$. Such an element $r$ is further called {\it strongly nil-clean} if the existing idempotent and nilpotent can be taken such that $be=eb$. Accordingly, a ring is called {\it nil-clean} (respectively, {\it strongly nil-clean}) if each of its elements is nil-clean (respectively, strongly nil-clean).

On the other side, following \cite{csj, chensjc}, \textit{strongly J-clean} rings are introduced as those rings in which each element can be written as the sum of an idempotent and an element from the Jacobson radical $J(R)$ that commute with each other.

In another vein, an element $a$ of a ring $R$ is said to be {\it quasi-nilpotent} if $1-ax$ is invertible for every $x \in R$ with $xa = ax$. It is, thereby, fairly obvious that both $Nil(R)$ and $J(R)$ are contained in $QN(R)$. It is also worth noting that quasi-nilpotent elements play a facilitating role in the investigation of the complicated structure of a Banach algebra $\mathcal{A}$.

Further, in 2025, Danchev {\it et al.} defined in \cite{daoa} the classes of {\it UQ rings} and {\it strongly quasi-nil clean rings}, that are non-trivial extensions of the classes of {\it UJ rings} and {\it strongly J-clean rings} (see, for a more concrete information, \cite{daoa}). Specifically, a ring is called {\it UQ} if $1 + U(R) = QN(R)$, and a ring $R$ is called {\it strongly quasi-nil clean} if every element of $R$ is the sum of an idempotent and a quasi-nilpotent element that commute with each other. It is proved there that $R$ is a {\it strongly quasi-nil clean ring} if, and only if, $R$ is both strongly clean and UQ.

In \cite{wang}, the set \[J^{\#}(R)=\{z\in R : z^n\in J(R)\ \text{for some } n\ge 1\}\] was introduced as the subset of \(R\) properly containing $J(R)$, which set, however, is {\it not} necessarily a subring.

Considering the stated above definitions above combined with the fact that \( J^{\#}(R) \) is a subset of \( R \), that is {\it not} necessarily an ideal, a logical question arises automatically: what is the structure of these rings \( R \) for which every element of $R$ can be written as the sum of an idempotent of $R$ and an element from \( J^{\#}(R) \) that commute with each other? This idea was first introduced in \cite{theory}.

\medskip

In this paper, in the context of our preceding discussion, we focus on the class of strongly \( J^{\#} \)-clean rings, where each element can be expressed as the sum of an idempotent and a commuting element from \(J^{\#}(R)\). We aim to study their crucial properties in detail and, besides, to understand their structure better by giving some relevant relationships with other well-known classes of rings. Concretely, we established the following achievements: In Theorem~\ref{reverse} we show that the structure of strong \( J^{\#}(R) \)-cleanness is somewhat symmetric. Next, in Theorem~\ref{1.6}, we give a close connections between the classes of clean rings and \( J^{\#}(R) \)-clean rings by posing only condition on the unit group of an arbitrary ring $R$. Further, we prove in Theorem~\ref{4} an equivalency between unique cleanness and abelian \( J^{\#} \)-cleanness. Likewise, studying group rings, we establish in Theorem~\ref{1.7} that if an arbitrary group ring $RG$ is strongly $J^{\#}$-clean, then so does $R$ and $G$ has to be a torsion $2$-group, as well as we demonstrate in the subsequent Proposition~\ref{group} that the reciprocal implication is somewhat complicated. Moreover, in Theorem~\ref{locstr}, we consider generalized matrix constructions over a local ring and obtain a necessary and sufficient condition in this topic. In closing, we succeed to formulate and confirm the truthfulness of two results, namely Theorems~\ref{main} and \ref{j}, in which we build certain key equivalencies between well-examined rings classes.

\section{Examples and Basic Properties}

We begin our work with the following technicalities, the first of which describes some general properties of the set \(J^{\#}(R)\).

\begin{lemma}\label{close prod}
Let $R$ be a ring.
	
(1) If $a\in J^{\#}(R)$ and $b \in J(R)$, then $a+b \in J^{\#}(R)$.

(2) $a\in J^{\#}(R)$ if, and only if, $-a\in J^{\#}(R)$.

(3) If $R = \prod R_i$, then $J^{\#}(R) = \prod J^{\#}(R_i)$.

(4) If $a\in J^{\#}(R)$, then $uau^{-1}\in J^{\#}(R)$ for each $u\in U(R)$.
\end{lemma}

\begin{proof}
(1) Suppose that $a \in J^{\#}(R)$ and $b \in J(R)$. Since $a^n \in J(R)$ for some $n\in \mathbb{N}$, we have:
\[
(a+b)^n = a^n + b^n + \sum_{\text{finite}} (\text{products of } a \text{ and } b).
\]
But, all these terms lie in $J(R)$, whence $a+b \in J^{\#}(R)$, as requested.

(2) This part is clear.

(3) This follows directly from the equality $J(R) = \prod J(R_i)$ and the behavior of powers in product rings.

(4) It is straightforward.
\end{proof}

Our basic instrument is the following concept.

\begin{definition}
We say that $R$ is a {\it strongly $J^{\#}$-clean} ring if every element of $R$ is the sum of an idempotent from $R$ and an element from $J^{\#}(R)$ that commute with each other.
\end{definition}

We proceed by proving the following.

\begin{lemma}\label{2}
(1) Suppose \( R = \prod_{i \in I} R_i \). Then, \( R \) is a strongly $J^{\#}$-clean ring if, and only if, for each \( i \in I \), \( R_i \) is a strongly $J^{\#}$-clean ring.
	
(2) Suppose \( R \) is a strongly $J^{\#}$-clean ring and \( I \) is an ideal of \( R \) such that \( I \subseteq J(R) \). Then, \( R/I \) is a strongly $J^{\#}$-clean ring (in particular, $R/J(R)$ is a strongly $J^{\#}$-clean ring).
\end{lemma}

\begin{proof}
(1) This follows directly from Lemma~\ref{close prod}(3).

(2) Let $a+I \in R/I$ be arbitrary. Thus, $a \in R$, and we can write $a = e + j$ as a strongly $J^{\#}$-clean decomposition. So,
\[
a+I = (e+I) + (j+I),
\]
where $e+I \in Id(R/I)$ and $(e+I)(j+I) = (j+I)(e+I)$.
It remains to show that $j+I \in J^{\#}(R/I)$. Since $j \in J^{\#}(R)$, we deduce $j^n \in J(R)$ for some $n$. Therefore,
\[
(j+I)^n = j^n + I \in J(R) + I \subseteq J(R) + J(R) \subseteq J(R).
\]
Finally, $j+I \in J^{\#}(R/I)$, as desired.
\end{proof}

\begin{proposition}\label{unit}
Let \( R \) be a ring. A unit \( u \in R \) is strongly $J^{\#}$-clean if, and only if, \( 1 - u \in J^{\#}(R) \).
\end{proposition}

\begin{proof}
To prove the necessity, assume \( u \in U(R) \) is strongly $J^{\#}$-clean. Then, \( u = e + j \), where \( e \in R \) is an idempotent, \( j \in J^{\#}(R) \) and \( ej = je \). Since \( uj = ju \), it follows that \( 1 - u^{-1}j \in U(R) \), and hence \( e = u(1 - u^{-1}j) \in U(R) \). Therefore, \( e = 1 \), and so \( 1 - u = -j \in J^{\#}(R) \), completing the argument.	
\end{proof}

\begin{corollary}\label{cor1}
If $R$ is strongly $J^{\#}$-clean, then $U(R)=1+J^{\#}(R)$.
\end{corollary}

\begin{proof}
It is apparent that $1 + J^{\#}(R) \subseteq U(R)$. The reverse inclusion follows directly from Proposition \ref{unit}, as required.
\end{proof}

\begin{corollary}\label{3}
If the ring \( R \) is strongly $J^{\#}$-clean, then $2\in J(R)$.
\end{corollary}

\begin{proof}
By Corollary \ref{cor1}, we get $U(R) = 1 + J^{\#}(R)$. Thus, $-1 \in 1 + J^{\#}(R)$, which implies $-2 \in J^{\#}(R)$. But, it follows from Lemma \ref{close prod}(2) that $2 \in J^{\#}(R)$. Since $2$ is central, for any $r \in R$, we infer $2r = r2 \in J^{\#}(R)$. Hence, $1 - 2r \in U(R)$, and thus $2 \in J(R)$.
\end{proof}

The next statement is immediate.

\begin{example}
(1) Every element in $J^{\#}(R)$ is strongly $J^{\#}$-clean.
	
(2) $a\in R$ is strongly $J^{\#}$-clean if, and only if, $1-a\in R$ is strongly $J^{\#}$-clean.	
\end{example}

\begin{lemma}\label{1.5}
Let $R$ be a ring and $e \in Id(R)$. Then, the equalities $$J^{\#}(eRe)=eRe\cap J^{\#}(R)=eJ^{\#}(R)e$$ hold.
\end{lemma}

\begin{proof}
According to the well-known equalities $J(eRe)=eRe\cap J(R)=eJ(R)e$ and the definition of $J^{\#}(R)$, the claim follows.
\end{proof}

The next consequence is now immediately true.

\begin{corollary}\label{1}
Let \( j \in J^{\#}(R) \) and \( e^2 = e \in R \). If \( ej = je \), then \[ej \in J^{\#}(R)\cap eRe=J^{\#}(eRe).\]
\end{corollary}

For an element \( a \) of a ring \( R \), put \( \ell(a) := \{ x \in R \mid xa = 0 \} \) and \( r(a) := \{ x \in R \mid ax = 0 \} \).

\medskip

The next extra technicalities are also worthwhile for our further presentation.

\begin{lemma}\label{0}
Let \( R \) be a ring, and suppose \( a = e + j \) is a strongly $J^{\#}$-clean expression of \( a \) in \( R \). Then, \( \ell(a) \subseteq R(1 - e) \) and \( r(a) \subseteq (1 - e)R \).
\end{lemma}

\begin{proof}
We shall establish only \( \ell(a) \subseteq R(1 - e) \), because the inclusion for \( r(a) \) follows by analogy. To that end, take any \( x \in \ell(a) \), so \( xa = x(e + j) = 0 \). This gives \( xe = -xj = -xje \) and, multiplying both sides of these equalities by \( (1 + j) \) from the right, we obtain \( xe(1 + j) = 0 \). Since \( j \in J^{\#}(R) \), it follows that \( 1 + j \in U(R) \), which implies \( xe = 0 \). Therefore, \( x = x(1 - e) \), and thus \( x \in R(1 - e) \), as claimed.
\end{proof}

\begin{proposition}
Let \( R \) be a ring, and let \( e^2 = e \in R \). Then, an element \( a \in eRe \) is strongly $J^{\#}$-clean in \( R \) if, and only if, \( a \) is strongly $J^{\#}$-clean in \( eRe \).
\end{proposition}

\begin{proof}
Assume \( a \in eRe \) is strongly $J^{\#}$-clean in \( eRe \). Write \( a = f + j \), where \( f^2 = f \in eRe \), \( j \in J^{\#}(eRe) \) and \( fj = jf \). Clearly, \( f \) is an idempotent in \( R \). Consulting with Corollary \ref{1}, \( J^{\#}(eRe) \subseteq J^{\#}(R) \), so \( a \) is strongly $J^{\#}$-clean in \( R \), as needed.

Conversely, suppose \( a = g + j' \), where \( g^2 = g \in R \), \( j' \in J^{\#}(R) \) and \( gj' = j'g \). Since \( a = eae \), it follows that \( 1 - e \in \ell(a) \cap r(a) \). Now, Lemma \ref{0} works to get that \[ 1 - e \in R(1 - g) \cap (1 - g)R = (1 - g)R(1 - g) .\] This, in turn, means that \( eg = ge \), and \( ej' = e(a - g) = (a - g)e = j'e \). Note, moreover, that \( (ege)^2 = ege \in eRe \) and \( (ege)(ej'e) = (ej'e)(ege) \). Looking at Corollary \ref{1} again, we derive \( ej'e \in J^{\#}(R) \cap eRe = J^{\#}(eRe) \). Therefore, \( a = ege + ej'e \) is strongly $J^{\#}$-clean in \( eRe \), a srequired.
\end{proof}

An immediate consequence is the following one.

\begin{corollary}\label{corner}
Let $R$ be a strongly $J^{\#}$-clean ring and $e \in Id(R)$. Then, the corner subring $eRe$ is too a strongly $J^{\#}$-clean ring.
\end{corollary}

\begin{lemma}\label{matrix}
If \( e = 0 \), then \( e \neq u + v \) for all \( u, v \in U(eRe) \). In particular, any matrix ring \( M_n(R) \) with \( n \geq 2 \) cannot be strongly $J^{\#}$-clean.
\end{lemma}

\begin{proof}
Suppose on the contrapositive that \( e = u + v \) for some \( u, v \in U(eRe) \). Since \( eRe \) is strongly $J^{\#}$-clean by Corollary \ref{corner} and \( u \in U(eRe) \), it follows from Proposition \ref{unit} that \( e - u = v \in J^{\#}(eRe) \). Hence, \( v \in U(eRe) \cap J^{\#}(eRe) \), which is the pursued contradiction.

To sustain the second, partial, claim, consider the matrix ring \( M_2(R) \). Set
\[
f := \begin{pmatrix}1 & 0 \\ 0 & 1\end{pmatrix} \in M_2(R).
\]
Then, one knows that \( f^2= f \neq 0 \), and simply write that
\[
\begin{pmatrix}1 & 0 \\ 0 & 1\end{pmatrix} =
\begin{pmatrix}1 & -1 \\ -1 & 0\end{pmatrix} +
\begin{pmatrix}0 & 1 \\ 1 & 1\end{pmatrix},
\]
which expresses the identity of $M_2(R)$ as a sum of two units in \( fM_2(R)f \). Therefore, \( M_2(R) \) is, manifestly, not strongly $J^{\#}$-clean. Now, Corollary \ref{corner} employs to conclude that \( M_n(R) \) is not strongly $J^{\#}$-clean for all \( n \geq 2 \), as stated.
\end{proof}

A set $\{e_{ij} : 1 \le i, j \le n\}$ of non-zero elements of $R$ is said to be a {\it system of $n^2$ matrix units} if $e_{ij}e_{st} = \delta_{js}e_{it}$, where $\delta_{jj} = 1$ and $\delta_{js} = 0$ for $j \neq s$. In this case, $e := \sum_{i=1}^{n} e_{ii}$ is an idempotent of $R$ and $eRe \cong M_n(S)$, where $$S = \{r \in eRe : re_{ij} = e_{ij}r,~~\textrm{for all}~~ i, j = 1, 2, . . . , n\}.$$

Recall the standard terminology that a ring $R$ is said to be {\it Dedekind-finite} assuming $ab=1$ yields $ba=1$ for any $a,b\in R$. In other words, all one-sided inverses in the ring are two-sided.

\medskip

We, thus, come to the following.

\begin{lemma} \label{dedkind finite}
Every strongly $J^{\#}$-clean ring is Dedekind-finite.
\end{lemma}

\begin{proof}
Suppose the contrary that \( R \) is not Dedekind-finite. Then, there exist elements \( a, b \in R \) such that \( ab = 1 \) but \( ba \neq 1 \). Define \( e_{ij} := a^i(1 - ba)b^j \), and let \( e = \sum_{i=1}^n e_{ii} \). So, there exists a non-zero ring \( S \) such that \( eRe \cong M_n(S) \). However, with Lemma \ref{corner} in hand, the ring \( eRe \) is strongly \( J^{\#} \)-clean, implying that \( M_n(S) \) is also strongly \( J^{\#} \)-clean, thus contradicting Lemma \ref{matrix}.
\end{proof}

The next technical criterion is critical.

\begin{proposition}
Let \( R \) be a ring. Then, an element \( a \in R \) is strongly $J^{\#}$-clean if, and only if, there exists \( x \in R \) such that \( x^2 a = x \), \( ax = xa \), and \( a - ax \in J^{\#}(R) \).
\end{proposition}

\begin{proof}
$(\Rightarrow)$ Suppose \( a \in R \) is strongly $J^{\#}$-clean. Then, there exists an idempotent \( e^2 = e \in R \) such that \( a - e \in J^{\#}(R) \) and \( ae = ea \). Define \( x := (a + 1 - e)^{-1} e \). Since \( a \), \( e \) and \( (a + 1 - e)^{-1} e \) commute with each other, we have
\[
ax = xa = (a + 1 - e)^{-1} e \cdot a = (a + 1 - e)^{-1} \cdot e \cdot (a + 1 - e) = e.
\]
Thus, \( x^2 a = x \cdot ax = x \cdot e = x \), and \( a - ax = a - e \in J^{\#}(R) \).

$(\Leftarrow)$ Assume there exists \( x \in R \) such that \( x^2 a = x \), \( ax = xa \), and \( a - ax \in J^{\#}(R) \). Put \( f = ax \). Then, one checks that \[ f^2 = ax \cdot ax = ax \cdot xa = a x^2 a = ax = f ,\] so \( f \) is an idempotent. Also, a routine check shows that
\[
af = a \cdot ax = ax \cdot a = f \cdot a,
\]
so \( af = fa \), and \( a - f = a - ax \in J^{\#}(R) \). Therefore, \( a \) is strongly $J^{\#}$-clean in \( R \), as formulated.	
\end{proof}

We now have at our disposal all the machinery necessary to prove the following major statement.

\begin{theorem}\label{reverse}
Given a ring \( R \) and elements \( a, b \in R \), the strong $J^{\#}$-cleanness of \( ab \) guarantees that \( ba \) also exhibits the strongly $J^{\#}$-clean structure.
\end{theorem}

\begin{proof}
Set \( \alpha := ab \) and \( \beta := ba \). Then, one verifies that \( a\beta = \alpha a \) and \( \beta b = b\alpha \). Since \( \alpha = ab \) is strongly $J^{\#}$-clean in \( R \), there exists \( x \in R \) such that \( x^2 \alpha = x \), \( \alpha x = x \alpha \), and \( \alpha - \alpha x \in J^{\#}(R) \).

We aim to show that \( \beta = ba \) is also strongly $J^{\#}$-clean. In fact, letting \( y = b x^2 a \), then one computes that
\[
\beta y = ba \cdot b x^2 a = b \alpha x \cdot x a = b x \alpha \cdot x a = b x^2 \cdot \alpha a = b x^2 a \cdot ba = y \beta,
\]
and
\[
y^2 \beta = b x^2 a \cdot b x^2 a \cdot ba = b x^2 \alpha \cdot x^2 \alpha \cdot a = b x^2 a = y.
\]

Now, we verify that \( \beta - \beta y \in J^{\#}(R) \). Indeed, one first noticing that
\[
\beta - \beta y = \beta - \beta \cdot b x^2 a = \beta - b \alpha x^2 a = \beta - b x \alpha a = \beta - b x a = b (1 - x) a.
\]

We next claim that, if \( ab \in J^{\#}(R) \), then \( ba \in J^{\#}(R) \).

\medskip

\textbf{Claim:} If \( ab \in J^{\#}(R) \), then \( ba \in J^{\#}(R) \).

\medskip

To inspect this implication, suppose \( ab \in J^{\#}(R) \). Then, there exists \( n \in \mathbb{N} \) such that \( (ab)^n \in J(R) \), so \( (ba)^{n+1} = b(ab)^na \in J(R) \), as claimed.

Now, since \[ \alpha - \alpha x = ab - abx = ab(1 - x) \in J^{\#}(R) ,\] it follows at once from the Claim that \( \beta - \beta y = b(1 - x)a \in J^{\#}(R) \), as required to complete the whole verification.
\end{proof}

Our automatic consequence is the following one.

\begin{corollary}
Let \( R \) be a ring and \( a, b \in R \). If \( 1 - ab \) is strongly $J^{\#}$-clean, then \( 1 - ba \) is also strongly $J^{\#}$-clean.
\end{corollary}

Our next crucial assertion is the following.

\begin{theorem}\label{1.6}
Suppose that $R$ is an arbitrary ring. Then, the following three conditions are equivalent:
	
(1) $R$ is a clean ring with $U(R)=1+J^{\#}(R)$ .
	
(2) $R$ is a $J^{\#}$-clean ring with $U(R)=1+J^{\#}(R)$.
	
(3) $R$ is a $J^{\#}$-clean ring.
\end{theorem}

\begin{proof}
(1) $\Rightarrow$ (2). Assume $a \in R$ is clean, so $a = e + u$ is a clean decomposition. Since $U(R) = 1 + J^{\#}(R)$, we can write $u = 1 + j$ with $j \in J^{\#}(R)$. Thus,
\[
a = (1 - e) + (2e + j).
\]
Appealing to Corollary \ref{3}, we have $2e \in J(R)$ forcing $2e + j \in J^{\#}(R)$ with the aid of Lemma \ref{close prod}(1), as required.

(2) $\Rightarrow$ (3). This implication is pretty clear.

(3) $\Rightarrow$ (1). Assume $R$ is $J^{\#}$-clean. Then, in view of Proposition \ref{unit}, we know $U(R) = 1 + J^{\#}(R)$. Furthermore, if $a \in R$ is $J^{\#}$-clean and $a = e + j$ is a $J^{\#}$-clean decomposition, then
\[
a = (1 - e) - \big(1 - (2e + j)\big).
\]
As in the previous case, we can easily verify that $1 - (2e + j) \in U(R)$, as requested.
\end{proof}

It is worthwhile mentioning that, by what we have proven in point (3), the condition $U(R)=1+J^{\#}(R)$ is {\it not} essential (in point (2)) in order to give a transversal between $J^{\#}$-cleanness and the ordinary cleanness. 

\medskip

The next statement is worthy of recording.

\begin{lemma}\label{cor2}
Any strongly $J^{\#}$-clean ring is strongly clean.
\end{lemma}

\begin{proof}
Let $a \in R$ be strongly $J^{\#}$-clean. Then, there exist $e^2 = e \in R$ and $j \in J^{\#}(R)$ such that $a = e + j$ and $ej = je$. Thus, we write $a = (1 - e) + (2e - 1 + j)$. Now, with Corollary \ref{3} at hand, we have $2 \in J(R)$, so $$2e -1 + j \in J(R) + U(R) \subseteq U(R),$$ as expected.
\end{proof}

What we additionally obtain is the following necessary and sufficient condition.

\begin{corollary}\label{33}
A ring $R$ is strongly $J^{\#}$-clean if, and only if, the following two points are fulfilled:
	
(1) $U(R)=1+J^{\#}(R)$.
	
(2) $R$ is strongly clean.
\end{corollary}

\begin{proof}
If $R$ is a strongly $J^{\#}$-clean ring, the claim is concluded from Corollary \ref{cor1} and Lemma \ref{cor2}.

For the converse, choose $a\in R$. Since $R$ is strongly clean, we write $-a=e+u$, where $e\in Id(R)$ and $u\in U(R)$ with $eu=ue$. However, by hypothesis, $u=1-j$ for some $j\in J^{\#}(R)$. Thus, $a=(1-e)+(-j)$ is a strongly $J^{\#}$-clean decomposition for $a$. So, $R$ is a strongly $J^{\#}$-clean ring, as promised.
\end{proof}

The following assertion is somewhat a bit surprising.

\begin{lemma}\label{333}
If $U(R)=1+J^{\#}(R)$, then $R/J(R)$ is a UU ring.
\end{lemma}

\begin{proof}
Suppose that $u + J(R) \in U(R / J(R))$. Thus, $u \in U(R)$, so we can write $u = 1 + j$ with $j \in J^{\#}(R)$. Therefore,
\[
u + J(R) = (1 + j) + J(R) = (1 + J(R)) + (j + J(R)),
\]
where $j + J(R) \in \operatorname{Nil}(R / J(R))$, because $j \in J^{\#}(R)$ gives $j^n \in J(R)$ for some $n$.
\end{proof}

As it is well-known, a ring $R$ is called {\it Boolean} if every element in $R$ is idempotent.

\begin{proposition}\label{booli}
If $R$ is a strongly $J^{\#}$-clean ring, then $R/J(R)$ is Boolean.
\end{proposition}

\begin{proof}
An appeal to Corollary \ref{33} and Lemma \ref{333}, combined with \cite[Theorem 4.3]{DL}, allows to deduce that $R/J(R)$ is really a Boolean ring.
\end{proof}

\begin{example} \label{example 1}
Let \( R \) be an arbitrary ring. Then, \( R[x] \) is {\it not} a strongly $J^{\#}$-clean ring.
\end{example}

\begin{proof}
If we assume the opposite that \( R[x] \) is strongly \( J^{\#} \)-clean, then it must also be strongly clean, which leads to a contradiction with well-established facts.
\end{proof}

The next two propositions are quite welcome.

\begin{proposition}
Let \( R \) be a ring. The following two statements are equivalent:

(1) \( R \) is strongly $J$-clean.

(2) \( R \) is strongly $J^{\#}$-clean and \( J^{\#}(R) = J(R)\).
\end{proposition}

\begin{proof}
\((2) \Rightarrow (1)\). It is evident.

\((1) \Rightarrow (2)\). The strong $J^{\#}$-cleanness of \( R \) is immediate. It just remains to verify that \( J^{\#}(R) \subseteq J(R) \). So, for any \( a \in J^{\#}(R) \), we write by assumption that \( a = e + j \), where \( e^2 = e \in R \), \( j \in J(R) \) and \( ej = je \). Thus, \( e = a - j \in J^{\#}(R) \) by Lemma \ref{close prod}(1) and, consequently, \( e \in J(R) \). Hence, \( e = 0 \), and thus it follows that \( a = j \in J(R) \), as asked for.
\end{proof}

\begin{proposition}\label{jj}
Let \( R \) be a ring. The following two assertions are equivalent:
	
(1) \( R \) is strongly nil-clean.
	
(2) \( R \) is strongly $J^{\#}$-clean and \( J^{\#}(R) = Nil(R)\).
\end{proposition}

\begin{proof}
\((2) \Rightarrow (1)\). It is straightforward.
	
\((1) \Rightarrow (2)\). The strong $J^{\#}$-cleanness of \( R \) is very clear. Next, we only need to show that \( J^{\#}(R) \subseteq Nil(R) \). To that target, given any \( a \in J^{\#}(R) \), we write by hypothesis that \( -a = e + n \), where \( e^2 = e \in R \), \( n^k = 0 \) for some integer \( k \), and \( en = ne \). Thus,
\[
0 = (-n)^k = (e + a)^k.
\]
Furthermore, since \( ea = ae \), exploiting the binomial theorem we have
\[
e = -\left( a^k + \binom{k}{1} ea^{k-1} + \cdots + \binom{k}{k-1} ea \right)
= -a\left( a^{k-1} + \binom{k}{1} ea^{k-2} + \cdots + \binom{k}{k-1} e \right).
\]
It now follows that
\[
1 - e = 1 + a\left( a^{k-1} + \binom{k}{1} ea^{k-2} + \cdots + \binom{k}{k-1} e \right).
\]
Hence, \( 1 - e \in U(R) \), so $e=0$, and thus \( a = -n \in Nil(R) \), as wanted.
\end{proof}

We are now preparing to prove the following central equivalencies.

\begin{theorem}\label{4}
Let \( R \) be a ring. The following three items are equivalent:

(1)  \( R \) is abelian $J^{\#}$-clean.

(2) For every \( a \in R \), there exists a unique idempotent \( e^2 = e \in R \) such that \( a = e + j \), where \( j \in J^{\#}(R) \) (i.e., $R$ is uniquely $J^{\#}$-clean).

(3)  \( R \) is uniquely clean.
\end{theorem}

\begin{proof}
(1) \( \Rightarrow \) (2). Let \( a \in R \). Suppose there exist idempotents \( e^2 = e \in R \) and \( f^2 = f \in R \) such that \( a = (1 - e) + j \) and \( a = (1 - f) + j' \), where \( j, j' \in J^{\#}(R) \). Then, \( a + e = 1 + j \in U(R) \), and \( af = jf \in J^{\#}(R) \). Since \( R \) is abelian, the elements \( a, e, j, f \) and \( j' \) commute pairwise. Hence, it must be that
\[
(1 - e)f = (1 - e)(a + e)(a + e)^{-1}f = (1 - e)a(a + e)^{-1}f = (1 - e)f(a + e)^{-1}af,
\]
which insures \( (1 - e)f(1 - (a + e)^{-1}af) = 0 \). But, because of the relationships, \( (a + e)^{-1}af = af(a + e)^{-1} \) and \( af \in J^{\#}(R) \), it follows that \( 1 - (a + e)^{-1}af \in U(R) \). Therefore, \( (1 - e)f = 0 \). In a way of similarity, we obtain \( (1 - f)e = 0 \). Finally, \( f = ef = fe = e \).

(2) \( \Rightarrow \) (3). Since \( R \) is \( J^{\#} \)-clean, and consuming Theorem \ref{1.6}, we have that \( R \) is clean and \( U(R) = 1 + J^{\#}(R) \). Let \( a \in R \). We may suppose that \( -a = g + u = f + v \), where \( g, f \) are idempotents, and \( u, v \) are units in \( R \). Thus, we can write \( u = -1 - j \) and \( v = -1 - j' \) for some \( j, j' \in J^{\#}(R) \). Consequently,
\[
a = -g - u = (1 - g) + j \quad \text{and} \quad a = -f - v = (1 - f) + j'.
\]
By hypothesis, it follows that \( g = f \) and, therefore, \( R \) is uniquely clean.

(3) \( \Rightarrow \) (1). In virtue of \cite[Lemma 4]{nzu}, \( R \) is abelian. Since \( R \) is uniquely clean, \( R/J(R) \) is Boolean knowing \cite[Theorem 20]{nzu}. Hence, \[ U(R) = 1 + J(R) \subseteq 1 + J^{\#}(R) \subseteq U(R) ,\] and so \( U(R) = 1 + J^{\#}(R) \). Finally, \( R \) is abelian \( J^{\#} \)-clean owing to Theorem \ref{1.6}.
\end{proof}

As a valuable consequence, we extract:

\begin{corollary}
Let \( R \) be a ring. The following three points are equivalent:

(1) \( R \) is local and  $J^{\#}$-clean.

(2) \( R \) is $J^{\#}$-clean with only trivial idempotents.

(3)  \( R/J(R) \cong \mathbb{Z}_2 \).	
\end{corollary}

\begin{proof}
(1)\( \Rightarrow \) (2). It is apparent.
	
(2) \( \Rightarrow \) (3). Since \( 0 \) and \( 1 \) are the only idempotents of \( R \), it follows that \( R \) is abelian. By Theorem \ref{4}, \( R \) is uniquely clean. The rest of the result now  follows from \cite[Theorem 15]{nzu}.
	
(3) \( \Rightarrow \) (1). This follows combining \cite[Theorem 15]{nzu} and Theorem \ref{4}.
\end{proof}

We continue by proving the following two final statements.

\begin{lemma}
Let $R$ be a ring. Then, the following four conditions hold:
	
(1) A division ring $R$ is strongly $J^{\#}$-clean if, and only if, $R \cong \mathbb{Z}_2$.
	
(2) A ring $R$ is local and strongly $J^{\#}$-clean if, and only if, $R/J(R) \cong \mathbb{Z}_2$.
	
(3) A semi-simple ring $R$ is strongly $J^{\#}$-clean if, and only if, $R \cong \mathbb{Z}_2 \times \cdots \times \mathbb{Z}_2$.
	
(4) If $R$ is semi-local and strongly $J^{\#}$-clean, then $R/J(R) \cong \mathbb{Z}_2 \times \cdots \times \mathbb{Z}_2$.
\end{lemma}

\begin{proof}
(1) If $R$ is a division ring, then we know that $J^{\#}(R)=0$ and also that $Id(R)=\{0,1\}$. Therefore, $U(R)={1}$. The reciprocality is obvious.
	
(2) If $R$ is local and strongly $J^{\#}$-clean ring, then (1) works. Reciprocally, suppose that $R/J(R) \cong \mathbb{Z}_2$. Then, according to \cite[Theorem 15]{nzu}, $R$ is an uniquely clean ring and, therefore, strongly $J^{\#}$-clean.
	
(3) If $R$ is a semi-simple ring, in accordance with the famous Wedderburn-Artin theorem, we write $R\cong \prod M_{n_i}(D_i)$, where all $D_i$ are division rings. Consequently, based on Lemma \ref{2}(1), Lemma \ref{matrix} and (1), we may freely conclude that $R \cong \mathbb{Z}_2 \times \cdots \times \mathbb{Z}_2$, as expected. The opposite implication is rather obvious, so its evidence is avoided.
	
(4) It is easy for establishment taking in mind (3).
\end{proof}

\begin{lemma} \label{ni idempotent}
Let $R$ be a strongly $J^{\#}$-clean ring. Then, $R$ is local if, and only if, it has no non-trivial idempotents.
\end{lemma}

\begin{proof}
If $R$ is local, then obviously $R$ has no non-trivial idempotents. Now, the treat the reverse, suppose that $R$ has no non-trivial idempotents. Then, for every $a \in R$, we have either $a \in J^{\#}(R)$ or $a-1 \in J^{\#}(R)$, which, in turn, ensures that $R = J^{\#}(R) \cup U(R)$. Moreover, it can elementarily be proven that $R$ is local if, and only if, $R = J^{\#}(R) \cup U(R)$, as we need.
\end{proof}

\section{Group Rings}

Let $R$ be an arbitrary ring and $G$ an arbitrary group. The notation $RG$ commonly refers to the \textit{group ring} of $G$ over $R$. In this context, the symbol $\Delta(RG)$ stands for the kernel of the classical \textit{augmentation map} $\varepsilon: RG \to R$, which is defined by the equality
\[
\varepsilon\left(\sum_{g \in G} a_g g\right) = \sum_{g \in G} a_g.
\]
This kernel is known as the \textit{augmentation ideal} of $RG$.

\medskip

Furthermore, a group $G$ is called a \textit{$p$-group} if every element of $G$ has finite order that is a power of a given prime number $p$. Moreover, a group $G$ is said to be \textit{locally finite} if each finitely generated subgroup of $G$ is finite.

\medskip

Our first instrument necessary to attack the chief result is the following.

\begin{lemma}\label{prop group ring}
Let $R$ be a strongly $J^{\#}$-clean ring, and let $G$ be a locally finite $2$-group. Then, the following three statements hold:

(1) $\Delta(RG) \subseteq J(RG)$.
		
(2) The quotient ring $RG/J(RG)$ is Boolean.
		
(3) $J(RG) = \{x \in RG : \varepsilon(x) \in J(R)\}$.
\end{lemma}

\begin{proof}
(1) Corollary~\ref{3} informs us that $2 \in J(R)$. Thus, applying \cite[Lemma 2]{zucc}, it follows that $\Delta(RG) \subseteq J(RG)$.
	
(2) Let $x = \sum_{g \in G} a_g g \in RG$. Then, we can write
\[
	x = \sum_{g \in G} -a_g(1 - g) + \sum_{g \in G} a_g \in \Delta(RG) + R.
\]
Hence, by (1), it follows that $RG = J(RG) + R$. Moreover, \cite[Proposition 9]{conel} applies to get that $J(RG) \cap R = J(R)$, so that
\[
	RG/J(RG) \cong (R + J(RG))/J(RG) \cong R/J(R).
\]
Since $R/J(R)$ is Boolean invoking Proposition \ref{booli}, we conclude that $RG/J(RG)$ is Boolean as well.
	
(3) Define the set $A := \{x \in RG : \varepsilon(x) \in J(R)\}$. As $\varepsilon$ is surjective and $J(R)$ is an ideal, we deduce $\varepsilon(J(RG)) \subseteq J(R)$, thus giving $J(RG) \subseteq A$.
	
Now, take $x = \sum_{g \in G} a_g g \in A$, so $\varepsilon(x) \in J(R)$. From part (2), we know that $RG/J(RG)$ is Boolean, whence $1 - g \in J(RG)$ for every $g \in G$. Therefore,
\[
	x = \sum_{g \in G} a_g(g - 1) + \sum_{g \in G} a_g = \sum_{g \in G} a_g(g - 1) + \varepsilon(x) \in J(RG),
\]
which shows that $A \subseteq J(RG)$. This proves $J(RG) = A$, as claimed.

Finally, since $G$ is locally finite, the usage of \cite[Proposition 9]{conel} justifies the identity $J(RG) \cap R = J(R)$, as stated.
\end{proof}

We now have all the machinery needed to prove the following main result of this section.

\begin{theorem}\label{1.7} Suppose that $R$ is a ring and $G$ is a group. If the group ring $RG$ is strongly $J^{\#}$-clean, then $R$ is strongly $J^{\#}$-clean and $G$ is a $2$-group.
\end{theorem}

\begin{proof}
One sees that Lemma~\ref{prop group ring} guarantees the inclusion $\varepsilon(RG) \subseteq J(RG)$. Now, since both relations $RG/\Delta(RG) \cong R$ and $\Delta(RG)\subseteq J(RG)$ are true, Lemma~\ref{2} discovers that the quotient $RG/J(RG)$ induces a strongly $J^{\#}$-clean structure on $R$. Therefore, $R$ is strongly $J^{\#}$-clean.
	
On the other hand, referring to \cite[Proposition 15]{conel}, we derive that $G$ must be a $p$-group for some prime $p \in J(R)$. However, Corollary~\ref{3} asserts that $2 \in J(R)$, and since two distinct primes cannot both lie in the Jacobson radical of a ring, it follows at once that $p = 2$. Hence, $G$ is a $2$-group.
\end{proof}

We continue our work with the following.

\begin{lemma}
Let $R$ be a semi-abelian ring, and let $G$ be a locally finite group. Then, the following two statements are equivalent:

(1) $RG$ is a strongly $J^{\#}$-clean ring.

(2) $R$ is a strongly $J^{\#}$-clean ring and $G$ is a $2$-group.
\end{lemma}

\begin{proof}
(1) $\Rightarrow$ (2). This implication follows directly from Theorem~\ref{1.7}.
	
(2) $\Rightarrow$ (1). Assume that $R$ is a strongly $J^{\#}$-clean ring. Thus, for every prime ideal $P$ of $R$, we have the isomorphism
\[
	R/P \cong (R/J(R))/(P/J(R)).
\]
Since $R/J(R)$ is Boolean, it follows that each residue field $R/P$ is Boolean too. In particular, the factor-ring $R/P$ is a Boolean ring and, according to \cite[Theorem 3.3]{wchen}, this quotient-ring is also local. So, $R/P$ must be isomorphic to the field $\mathbb{F}_2$.
	
But, utilizing \cite[Theorem 3.3]{chin}, we conclude that the group ring $RG$ is strongly $\pi$-regular and thus strongly clean. Furthermore, since $RG/J(RG) \cong R/J(R)$ is Boolean, it follows that $RG$ is a UJ-ring. Now, as Corollary~\ref{33} claims, $RG$ is strongly $J^{\#}$-clean, as desired.
\end{proof}

It follows from the preceding lemma that, if $R$ is simultaneously an abelian and strongly $J^{\#}$-clean ring, and $G$ is a locally finite $2$-group, then the group ring $RG$ is also strongly $J^{\#}$-clean.

Nevertheless, we now proceed to establish this result through an alternative approach in what follows.

\begin{lemma}
Let $R$ be an abelian ring, and let $G$ be a locally finite group. Then, the following two issues are equivalent:

(1) $RG$ is a strongly $J^{\#}$-clean ring.

(2) $R$ is a strongly $J^{\#}$-clean ring and $G$ is a $2$-group.
\end{lemma}

\begin{proof}
(1) $\Rightarrow$ (2). This implication follows directly from Theorem~\ref{1.7}.
	
(2) $\Rightarrow$ (1). Following absolutely the same argument as in part (2) of Lemma~\ref{prop group ring}, we observe that $RG = J(RG) + R$.
	
Now, take any $f \in RG$. Without loss of generality, we may write $f = a + j$ with $a \in R$ and $j \in J(RG)$. Since $R$ is strongly $J^{\#}$-clean, we can express $a = e + x$, where $e \in \mathrm{Id}(R)$ and $x \in J^{\#}(R)$.
	
Now, Lemma~\ref{prop group ring} works to receive $x(1 - x) \in J(RG)$ and, as $x \in J^{\#}(R)$, it follows that $1 - x \in U(R)$, whence $x \in J(RG)$. Consequently,
\[
	f = a + j = (e + x) + j = e + (x + j) \in \mathrm{Id}(RG) + J(RG).
\]
Furthermore, since $R$ is abelian, we have $e(x + j) = (x + j)e$. This substantiates that $e$ and $x + j$ commute, completing the strongly $J^{\#}$-clean decomposition of $f$. Therefore, $RG$ is a strongly $J^{\#}$-clean ring, as claimed.
\end{proof}

We terminate this section with the following assertion, that is relevant to Theorem~\ref{1.7} listed above.

\begin{proposition}\label{group}
Let $R$ be a ring, and let $G \neq 1$ be a locally finite group in which each finite subgroup has odd order. Then, the group ring $RG$ is not strongly $J^{\#}$-clean.
\end{proposition}

\begin{proof}
Assume, to the contrary, that $RG$ is strongly $J^{\#}$-clean. Thus, Theorem~\ref{1.7} refers us to the fact that $R$ is also strongly $J^{\#}$-clean. In particular, Proposition \ref{booli} is working to infer that the factor ring $\bar{R} := R / J(R)$ is Boolean. But, since every finite subgroup of $G$ has odd order, it follows that all natural numbers $n \in |G|$ (i.e., the orders of finite subgroups of $G$) are units in $R$. Resultantly, each such $n$ remains a unit in $\bar{R}$.
	
Now, observe that $\bar{R}$ is a Boolean ring, and hence von Neumann regular. Since $G$ is locally finite, we can apply \cite[Theorem 3]{conel} to deduce that the group ring $\bar{R}G$ is also von Neumann regular. On the other hand, as a homomorphic image of the strongly $J^{\#}$-clean ring $RG$, the ring $\bar{R}G$ must also be strongly $J^{\#}$-clean. Therefore, $\bar{R}G$ is necessarily Boolean, since all strongly $J^{\#}$-clean rings with identity are Boolean. However, this leads to an obvious contradiction, because being a Boolean group ring readily implies that the former group must be trivial, that is, $G = 1$, which contradicts the assumption that $G \neq 1$. Thus, in turn, $RG$ cannot be strongly $J^{\#}$-clean, as asserted.
\end{proof}

\section{Generalized Matrix Rings}

In virtue of Lemma \ref{matrix}, the matrix ring $M_n(R)$ is {\it not} always strongly $J^{\#}$-clean for $n \ge 2$. Motivated by this observation, we characterize in this section when a generalized matrix ring over a local ring is strongly $J^{\#}$-clean.

\medskip

Let $R$ be a ring, and let $s \in R$ be central. We use $K_s(R)$ to denote the set
\[
\{ [a_{ij}] \in M_2(R) \mid a_{ij} \in R \text{ for all } i,j \}
\]
equipped with the following operations:

\noindent Addition:
\[
\begin{bmatrix} a & b \\ c & d \end{bmatrix}
+
\begin{bmatrix} a' & b' \\ c' & d' \end{bmatrix}
=
\begin{bmatrix} a+a' & b+b' \\ c+c' & d+d' \end{bmatrix},
\]

\noindent Multiplication:
\[
\begin{bmatrix} a & b \\ c & d \end{bmatrix}
\begin{bmatrix} a' & b' \\ c' & d' \end{bmatrix}
=
\begin{bmatrix}
	aa' + s\,b\,c' & a b' + b d' \\
	c a' + d c' & s\,c\,b' + d d'
\end{bmatrix}.
\]

The element $s$ is called the {\it multiplier} of $K_s(R)$. With these operations in hands, $K_s(R)$ manifestly forms a ring. For more details, we refer to \cite{tang}.

\medskip

Our preliminary technical claims are these:

\begin{lemma}\cite{kry}.
Let $R$ be a ring with $s \in R$ central, and let $J = J(R)$ be the Jacobson radical of $R$. Then, the following two claims are valid:

(1) $J(K_s(R)) = \begin{bmatrix}
	J &  (s : J) \\
	(s : J) & J
\end{bmatrix}$, where $(s : J) = \{ r \in R \mid r s \in J \}$.

(2) If $R$ is a local ring with $s \in J$, then
\[
J(K_s(R)) = \begin{bmatrix}
	J & R  \\
	R & J
\end{bmatrix}
.\]

Moreover, $\begin{bmatrix} a & x \\ y & b \end{bmatrix} \in U(K_s(R))$ if, and only if, $a, b \in U(R)$.
\end{lemma}

Our own achievements here are the following.

\begin{proposition}\label{12}
Let $R$ be a ring and let $s \in R$ be central. Then, $A \in K_s(R)$ is strongly $J^{\#}$-clean if, and only if, there exists $P \in U(K_s(R))$ such that $PAP^{-1} \in K_s(R)$ is strongly $J^{\#}$-clean.
\end{proposition}

\begin{proof}
It is just enough to demonstrate that, if $A \in K_s(R)$ is strongly $J^{\#}$-clean, then so is $PAP^{-1}$.

To this goal, write $A = E + T$, where $E^2 = E$, $T \in J^{\#}(K_s(R))$, and $ET = TE$. Thus, one follows that
\[
PAP^{-1} = PEP^{-1} + PTP^{-1},
\]
where $(PEP^{-1})^2 = PEP^{-1}$, $PTP^{-1} \in J^{\#}(K_s(R))$ in conjunction with Lemma \ref{close prod}(4), and $PEP^{-1}$ commutes with $PTP^{-1}$.

Hence, $PAP^{-1}$ is strongly $J^{\#}$-clean in $K_s(R)$, as requested.
\end{proof}

\begin{lemma}\cite{gur}\label{222}.
Let $R$ be a local ring with $s \in R$ be central, and let $E$ be a non-trivial idempotent in $K_s(R)$. Then, the following two items are fulfilled:

(1) If $s \in U(R)$, then
$
E \sim \begin{bmatrix} 1 & 0 \\ 0 & 0 \end{bmatrix}.
$

(2) If $s \in J(R)$, then either
$
E \sim \begin{bmatrix} 1 & 0 \\ 0 & 0 \end{bmatrix} \quad \text{or} \quad E \sim \begin{bmatrix} 0 & 0 \\ 0 & 1 \end{bmatrix}.
$
\end{lemma}

\begin{proposition}\label{13}
Let $R$ be a local ring with $s \in R$ central. If
$
\begin{bmatrix} a & 0 \\ 0 & b \end{bmatrix} \in J^{\#}(K_s(R)),
$
then $a, b \in J(R)$.
\end{proposition}

\begin{proof}
Assume
$
\begin{bmatrix} a & 0 \\ 0 & b \end{bmatrix} \in J^{\#}(K_s(R)).
$
Since $R$ is local, either $a, b \in U(R)$, or one of the elements $a, b \in J(R)$.

If both $a, b \in U(R)$, then
$
\begin{bmatrix} a & 0 \\ 0 & b \end{bmatrix} \in U(K_s(R)),
$
which is a contradiction.

So, without loss of generality, we may assume $a \in J(R)$ and $b \in U(R)$. Then, considering the relations
\[
I - \begin{bmatrix} a & 0 \\ 0 & b \end{bmatrix} \begin{bmatrix} 0 & 0 \\ 0 & b^{-1} \end{bmatrix}
= \begin{bmatrix} 1 & 0 \\ 0 & 0 \end{bmatrix} \in U(K_s(R)),
\]
we find that they are impossible. Hence, we must have $a, b \in J(R)$, as formulated.
\end{proof}

We now have all the ingredients necessary to prove the following main result for this section.

\begin{theorem}\label{locstr}
Let $R$ be a local ring with $s \in R$ central. Then, $A \in K_s(R)$ is strongly $J^{\#}$-clean if, and only if, one of the following three statements holds:

(1) $A \in J^{\#}(K_s(R))$,

(2) $A \in I + J^{\#}(K_s(R))$, or

(3) $A \sim \begin{bmatrix} a & 0 \\ 0 & b \end{bmatrix}$ where one of $a, b$ is in $J(R)$ and the other is in $1 + J(R)$.
\end{theorem}

\begin{proof}
\noindent\((\Leftarrow)\) It just suffices to show that
$
	B := \begin{bmatrix} a & 0 \\ 0 & b \end{bmatrix}
$
is strongly $J^{\#}$-clean. So, with no harm in generality, we can assume that
\( a \in J(R) \) and \( b \in 1 + J(R) \). Thus,
\[
	B = \begin{bmatrix} a & 0 \\ 0 & b \end{bmatrix}
	= \begin{bmatrix} 0 & 0 \\ 0 & 1 \end{bmatrix}
	+ \begin{bmatrix} a & 0 \\ 0 & b - 1 \end{bmatrix},
\]
where
$
	\left( \begin{bmatrix} 0 & 0 \\ 0 & 1 \end{bmatrix} \right)^2
	= \begin{bmatrix} 0 & 0 \\ 0 & 1 \end{bmatrix}
$
and
\[
	\begin{bmatrix} a & 0 \\ 0 & b - 1 \end{bmatrix} \in J(K_s(R)) \subseteq J^{\#}(K_s(R)).
\]
Hence, \(B\) is strongly $J^{\#}$-clean.
	
\noindent\((\Rightarrow)\) Suppose that \(A\) is strongly $J^{\#}$-clean in \(K_s(R)\). Then, there exists \(E \in K_s(R)\) with \(E^2 = E\) such that \(A - E \in J^{\#}(K_s(R))\) and \(AE = EA\).
	
If, for a moment, \(E = 0\) or \(E = I\), then \(A \in J^{\#}(K_s(R))\) or \(A \in I + J^{\#}(K_s(R))\), respectively. Otherwise, via Lemma \ref{222}, either
\[
E \sim \begin{bmatrix} 1 & 0 \\ 0 & 0 \end{bmatrix}
\quad \text{or} \quad
E \sim \begin{bmatrix} 0 & 0 \\ 0 & 1 \end{bmatrix}.
\]
Assume only that \(E \sim \begin{bmatrix} 0 & 0 \\ 0 & 1 \end{bmatrix}\) as the other case is very similar. Then, there exists \(P \in U(K_s(R))\) such that
\[
	PEP^{-1} = \begin{bmatrix} 0 & 0 \\ 0 & 1 \end{bmatrix}.
\]
By Proposition \ref{12},
\[
	PAP^{-1} - PEP^{-1} = PTP^{-1}
\]
is strongly $J^{\#}$-clean decomposition in \(K_s(R)\). Let \(PTP^{-1} = [a_{ij}]\). Thus, from
\[
	PTP^{-1} \begin{bmatrix} 0 & 0 \\ 0 & 1 \end{bmatrix}
	= \begin{bmatrix} 0 & 0 \\ 0 & 1 \end{bmatrix} PQP^{-1},
\]
we obtain \(a_{12} = a_{21} = 0\), whence
\[
	PQP^{-1} = \begin{bmatrix} a_{11} & 0 \\ 0 & a_{22} \end{bmatrix}
	\in J^{\#}(K_s(R)).
\]
By Proposition \ref{13}, \(a_{11}, a_{22} \in J(R)\). Consequently,
\[
	PAP^{-1} = PEP^{-1} + PTP^{-1}
	= \begin{bmatrix} a_{11} & 0 \\ 0 & a_{22} + 1 \end{bmatrix},
\]
where \(a_{11} \in J(R)\) and \(a_{22} + 1 \in 1 + J(R)\), as required.
\end{proof}

\section{Certain rings that are Equivalent to Strongly J-Clean Rings}

Let $R$ be a ring. We say that an element $a \in R$ is \emph{$\Delta$-nilpotent} if $1-au$ is invertible for every $u \in U(R)$ satisfying $au=ua$. The set of all $\Delta$-nilpotent elements of $R$ is denoted by $\Delta N(R)$. Clearly, the inclusions $$J(R) \subseteq QN(R) \subseteq \Delta N(R)~ \bigvee ~ Nil(R) \subseteq QN(R) \subseteq \Delta N(R)$$ are always fulfilled. Thus, a ring $R$ is called a \emph{$\Delta NU$-ring} if $U(R)=1+\Delta N(R)$, where $U(R)$ denotes the group of units of $R$. This class of rings was recently introduced by Tien \emph{et al.}~\cite{tkq}. 

We now define the class consisting of \emph{strongly $\Delta$-nil-clean rings}, that are such rings $R$ for which each element of $R$ can be expressed as $a=e+q$ with $e^2=e$, $q \in \Delta N(R)$ and $eq=qe$.

\medskip

Our first basic preliminaries in this direction are these:

\begin{lemma}\label{111}
Each strongly $\Delta$-nil-clean ring is $\Delta NU$.
\end{lemma}

\begin{proof}
Since \(1 + \Delta N(R) \subseteq U(R)\), it just remains to verify the opaque containment \(U(R) \subseteq 1 + \Delta N(R)\). To verify that, take an arbitrary unit \(u \in U(R)\). By hypothesis, we can write \(u = e + q\) for some idempotent \(e \in Id(R)\) and \(q \in \Delta N(R)\) such that \(eq = qe\). Multiplying both sides by \(u^{-1}\) on the left, we get \(u^{-1}e = 1 - u^{-1}q \in U(R)\) for some \(q \in \Delta N(R)\). This means that \(e\) is, simultaneously, a unit and an idempotent, ensuring automatically that \(e = 1\). Hence, \(u = 1 + q \in 1 + \Delta N(R)\), as asked.
\end{proof}

\begin{lemma}\label{11}
Any strongly $\Delta$-nil-clean ring is strongly clean.
\end{lemma}

\begin{proof}
Let $a \in R$ be strongly $\Delta$-nil-clean. Then, there exist $e^2 = e \in R$ and $q \in \Delta N(R)$ such that $a = e + q$ and $eq = qe$. Thus, we have $a = (1 - e) + (2e - 1 + q)$. With \cite[Theorem 3.10]{tkq} at hand, we have $2 \in J(R)$, so that $2e -1 + q \in J(R) + U(R) \subseteq U(R)$, as pursued.
\end{proof}

According to \cite[Lemma 2.1]{tkq}, the following inclusions hold:
\[
J(R) \subseteq QN(R) \subseteq \Delta N(R) \quad
\]
These nested relationships among different nilpotent-related subsets motivate the following unexpected result.

\begin{theorem}\label{main}
Let $R$ be a ring. Then the following are equivalent:
	
(1) $R$ is a strongly $J$-clean ring.
	
(2) $R$ is a strongly quasi-nil-clean ring.
	
(3) $R$ is a strongly $\Delta$-nil-clean ring.
\end{theorem}

\begin{proof}
Implications \( (1) \Rightarrow (2) \Rightarrow (3) \) follow directly from the chain of containments $J(R)\subseteq QN(R)\subseteq \Delta N(R)$, which obviously justifies the transition between the conditions.
	
$(3) \Rightarrow (1)$. A subsequent application of Lemmas \ref{11} and \ref{111} leads to this that the ring \( R \) is strongly clean and satisfies the \( \Delta NU \) condition. Consequently, \( R/J(R) \) is Boolean thanks to \cite[Corollary 4.14]{tkq}. It then follows from \cite[Theorem 2.3]{csj} that \( R \) is strongly \( J \)-clean, as wanted.
\end{proof}

As a consequence, we extract:

\begin{corollary}\label{8}
Let $R$ be a ring. Then, the following two claims are equivalent:
	
(1) $R$ is a uniquely clean ring.
	
(2) $R$ is an abelian $\Delta$-nil-clean ring.	
\end{corollary}

\begin{proof}
$(1)\Rightarrow(2)$. Assume that \( R \) is a uniquely clean ring. Appealing to \cite[Lemma 4]{nzu}, each idempotent in \( R \) is central, assuring that \( R \) is abelian. Furthermore, \cite[Theorem 20]{nzu} can be applied to get that, for every \( a \in R \), there exists a unique idempotent \( e \) such that \( a - e \in J(R) \subseteq \Delta N(R) \). Consequently, there exists some \( q \in \Delta(R) \) such that \( a = e + q \). This confirms that \( R \) is a $\Delta$-nil-clean ring.
	
$(2)\Rightarrow(1)$. Since \( R \) is an abelian $\Delta$-nil-clean ring, it follows that \( R \) is strongly $\Delta$-nil-clean. A consultation with Theorem~\ref{main} discovers that \( R \) is strongly \( J \)-clean. Consequently, \( R \) is uniquely clean with the help of \cite[Corollary 2.4]{csj}, as desired.
\end{proof}

Years ago, Anderson and Camillo posed the logical question of whether or not a homomorphic image of a uniquely clean ring remains uniquely clean. This question was affirmatively answered by Nicholson and Zhou (see \cite[Theorem 22]{nzu}), and later on, an alternative approach was provided by Chen (see \cite[Corollary 2.4]{csj}). In what follows, we offer a new result in this aspect.

\medskip

Precisely, we now derive the following consequence.

\begin{corollary}
Let $R$ be a ring. Then, the following six conditions are equivalent:
	
(1) $R$ is a abelian $J$-clean ring.
	
(2) $R$ is a abelian $J^{\#}$-clean ring.
	
(3) $R$ is a abelian quasi-nil-clean ring.
	
(4) $R$ is a abelian $\Delta$-nil-clean ring.	
	
(5) $R$ is an uniquely clean ring.
	
(6) Every homomorphic image of $R$ is a uniquely clean ring.
\end{corollary}

\begin{proof}
The equivalence of conditions (1), (5) and (6) is established in \cite[Corollary 2.4]{csj}. In addition, conditions (4) and (5) are, evidently, equivalent looking at Corollary~\ref{8}, and the equivalence of (2) and (5) is ensured by Theorem~\ref{4}. Finally, knowing the validity of inclusions \( J(R) \subseteq QN(R) \subseteq \Delta N(R) \), the proof is over.
\end{proof}

We have previously illustrated in Proposition \ref{jj} that a ring \( R \) is strongly \( J \)-clean exactly when it is strongly \( J^{\#} \)-clean and satisfies the condition \( J(R) = J^{\#}(R) \). Remarkably, we now demonstrate that this characterization remains valid even in the absence of the equality \( J(R) = J^{\#}(R) \), thereby revealing a deeper structural alignment between strongly \( J \)-clean and strongly \( J^{\#} \)-clean rings showing their curious coincident.

\medskip

Surprisingly, we are able to prove the following equivalence.

\begin{theorem}\label{j}
Let $R$ be a ring. Then, the following two statements are equivalent:
	
(1) $R$ is a strongly $J$-clean ring.
	
(2) $R$ is a strongly $J^{\#}$-clean ring
\end{theorem}

\begin{proof}
Implication \( (1) \Rightarrow (2) \) is immediate, because we always have $J(R)\subseteq J^{\#}(R)$.
	
Implication $(2) \Rightarrow (1)$ follows at once from Proposition \ref{booli}, claiming that $R / J(R)$ is Boolean, and combining this with \cite[Theorem 2.3]{csj} guaranteeing that $R$ is, in fact, strongly $J$-clean, as promised.
\end{proof}

%\noindent{\bf Acknowledgement.} The authors express their sincere gratitude to the expert referee for the numerous competent suggestions made which lead to a substantial improvement of the exposition.

\medskip

\noindent{\bf Funding:} The work of the first-named author, P.V. Danchev, is partially supported by the project Junta de Andaluc\'ia under Grant FQM 264.

\end{document}